\documentclass{amsart}

\usepackage[shortlabels]{enumitem}
\usepackage{color}
\usepackage{tikz}
\usepackage{tikz-cd}
\usetikzlibrary{matrix}
\usepackage{hyperref}
\usepackage {cleveref}

\newtheorem{theorem}{Theorem}[section]
\newtheorem{proposition}[theorem]{Proposition}
\newtheorem{lemma}[theorem]{Lemma}

\theoremstyle{definition}
\newtheorem{definition}[theorem]{Definition}

\theoremstyle{remark}
\newtheorem{remark}[theorem]{Remark}
\newtheorem{example}[theorem]{Example}

\numberwithin{equation}{section}

\begin{document}
\title[Special Generic Maps of Rational Homology Spheres]{On Special Generic Maps of Rational Homology Spheres into Euclidean Spaces}

\author[D.J. Wrazidlo]{Dominik J. Wrazidlo}
\address{Institute of Mathematics for Industry, Kyushu University, Motooka 744, Nishi-ku, Fukuoka 819-0395, Japan; Phone: +497010632874; ORCID: 0000-0001-7265-1791}
\email{d-wrazidlo@imi.kyushu-u.ac.jp}
\thanks{The author has been supported by JSPS KAKENHI Grant Number JP18F18752.
This work was written while the author was a JSPS International Research Fellow (Postdoctoral Fellowships for Research in Japan (Standard)).}


\date{\today.}


\begin{abstract}
Special generic maps are smooth maps between smooth manifolds with only definite fold points as their singularities.
The problem of whether a closed $n$-manifold admits a special generic map into Euclidean $p$-space for $1 \leq p \leq n$ was studied by several authors including Burlet, de Rham, Porto, Furuya, \`{E}lia\v{s}berg, Saeki, and Sakuma.
In this paper, we study rational homology $n$-spheres that admit special generic maps into $\mathbb{R}^{p}$ for $p < n$.
We use the technique of Stein factorization to derive a necessary homological condition for the existence of such maps for odd $n$.
We examine our condition for concrete rational homology spheres including lens spaces and total spaces of linear $S^{3}$-bundles over $S^{4}$, and obtain new results on the (non-)existence of special generic maps.
\end{abstract}

\maketitle

\section{introduction}
Let $f \colon M^{n} \rightarrow \mathbb{R}^{p}$, $1 \leq p \leq n$, be a smooth map of a connected closed $n$-dimensional smooth manifold $M$ into Euclidean $p$-space.
A point $x \in M$ is called a \emph{definite fold point} of $f$ if there exist local coordinates $(x_{1}, \dots, x_{n})$ and $(y_{1}, \dots, y_{p})$ centered at $x$ and $f(x)$, respectively, such that $f$ takes the form
\begin{align*}\label{definition definite fold point}
y_{i} \circ f &= x_{i}, \qquad 1 \leq i \leq p-1, \\
y_{p} \circ f &= x_{p}^{2} + \dots + x_{n}^{2}.
\end{align*}
The map $f$ is called a \emph{special generic map} if every singular point of $f$ is a definite fold point.
In this paper, we study special generic maps of rational homology spheres, i.e., closed manifolds with the rational homology groups of a sphere.

The notion of special generic maps first appeared in a paper of Calabi \cite{cal} under the name of quasisurjective mappings.
We note that special generic maps $M^{n} \rightarrow \mathbb{R}$ are the same as Morse functions of $M$ with only maxima and minima as their critical points, and can thus only exist when $M^{n}$ is homeomorphic to the standard $n$-sphere $S^{n}$ by a well-known result of Reeb \cite{R}.
Special generic maps $M^{n} \rightarrow \mathbb{R}^{2}$ were studied for $n = 3$ by Burlet and de Rham \cite{BdR}, and for $n > 3$ by Porto and Furuya \cite{PF}, and Saeki \cite{S1}.
Moreover, Sakuma \cite{sak1} and Saeki \cite{S1} studied special generic maps $M^{n} \rightarrow \mathbb{R}^{3}$ under various assumptions on the source manifold $M^{n}$.
Hara \cite{hara} studied the existence of special generic maps $M^{n} \rightarrow \mathbb{R}^{p}$ for $p \leq n/2$ by using $L^{2}$-Betti numbers of $M^{n}$.
\`{E}lia\v{s}berg \cite{E} studied special generic maps $M^{n} \rightarrow \mathbb{R}^{n}$ for orientable $M^{n}$.

For a given source manifold $M^{n}$, Saeki posed the problem to determine the set $S(M^{n})$ of all integers $p \in \{1, \dots, n\}$ for which there exists a special generic map $M^{n} \rightarrow \mathbb{R}^{p}$ (see Problem 5.3 in \cite{S2}).
Saeki observed that $S(M^{n}) = \{1, \dots, n\}$ if and only if $M^{n}$ is diffeomorphic to the standard $n$-sphere $S^{n}$.
Moreover, we note that if $S(M^{n}) \not\subset \{n\}$, then $M$ is (oriented) nullcobordant by Corollary 3.3 in \cite{S1}.
Nishioka \cite{nis} determined the dimension set $S(M^{5})$ for any simply connected closed $5$-manifold $M$.
In \cite{W}, the author determined the dimension set $S(\Sigma^{7})$ for $14$ of Milnor's exotic $7$-spheres $\Sigma^{7}$.
\par\medskip

In this paper, we study the dimension set $S(M^{n})$ of a rational homology sphere $M$ of odd dimension $n$ by using the technique of Stein factorization of special generic maps (see \Cref{Stein factorization}).
Previously, Saeki \cite{S1} obtained the following characterization of homotopy spheres in terms of Stein factorization.

\begin{theorem}[Proposition 4.1 in \cite{S1}]\label{theorem homotopy spheres}
Let $f \colon M^{n} \rightarrow \mathbb{R}^{p}$ ($1 \leq p < n$) be a special generic map.
Then $M^{n}$ is a homotopy sphere if and only if the Stein factorization $W_{f}$ is contractible.
\end{theorem}

In \Cref{proof of main result}, we show the following homological version of \Cref{theorem homotopy spheres} for any coefficient ring $R \neq 0$ with identity.

\begin{theorem}\label{theorem homology spheres}
Let $f \colon M^{n} \rightarrow \mathbb{R}^{p}$ ($1 \leq p < n$) be a special generic map.
Suppose that $M^{n}$ is orientable.
If $M^{n}$ is an $R$-homology $n$-sphere (see \Cref{definition homology sphere and ball}), then the Stein factorization $W_{f}$ is an $R$-homology $p$-ball.
The converse holds under the additional assumption that $R$ is a principal ideal domain (for example, $R = \mathbb{Z}$ or $R = k$ a field).
\end{theorem}

We observe that \Cref{theorem homotopy spheres} is a consequence of \Cref{theorem homology spheres} for $R = \mathbb{Z}$.
In fact, this follows from the homological version of the Whitehead theorem (see Corollary 4.33 in \cite[p. 367]{H}) by noting that $M^{n}$ is simply connected if and only if the Stein factorization $W_{f}$ is simply connected (see Proposition 3.9 in \cite{S1}).

As an application of our \Cref{theorem homology spheres}, we show in \Cref{proposition torsion condition} that if a rational homology sphere $M$ of odd dimension $n = 2k+1 \geq 5$ admits a special generic map into $\mathbb{R}^{p}$ for some $1 \leq p < n$, then the cardinality of the finite abelian group $H_{k}(M; \mathbb{Z})$ is the square of an integer.
However, this is in general not a sufficient condition for the existence of special generic maps on $M$ (see \Cref{remark non-sufficiency}).
We point out that \Cref{proposition torsion condition} can be seen as a torsion analog of the fact that a closed manifold which admits a special generic map into Euclidean $p$-space for some $1 \leq p < n$ has even Euler characteristic (see e.g. Corollary 3.8 in \cite{S1}).
As shown in \Cref{proposition realization of torsion group}, the square of a positive integer can always be realized as $|H_{k}(M; \mathbb{Z})|$ for some highly connected rational homology sphere $M$ of suitable odd dimension $n = 2k+1 \geq 5$ that admits a special generic map into $\mathbb{R}^{p}$ for some $1 \leq p < n$.
On the other hand, there are plenty of rational homology $n$-spheres $M$ for which $|H_{k}(M; \mathbb{Z})|$ is \emph{not} the square of an integer, so that $M$ admits no special generic maps into $\mathbb{R}^{p}$ for any $1 \leq p < n$.
For instance, we show that this is the case for many lens spaces whose dimension is congruent to $3$ (mod $4$) (see \Cref{lens spaces}), and many total spaces of linear $S^{3}$-bundles over $S^{4}$ (see \Cref{s3 bundles over s4}).

\subsection*{Notation}
The cardinality of a set $X$ is denoted by $|X|$.
The symbol $\cong$ either means diffeomorphism of smooth manifolds or isomorphism of groups.
The singular locus of a smooth map $f$ between smooth manifolds will be denoted by $S(f)$.
Let $D^{p} = \{x = (x_{1}, \dots, x_{p}) \in \mathbb{R}^{p}; x_{1}^{2} + \dots + x_{p}^{2} \leq 1\}$ denote the closed unit ball in Euclidean $p$-space, and $S^{p-1} := \partial D^{p}$ the standard $(p-1)$-sphere.

\section{Preliminaries}

In preparation of the proofs of our results, we review in this section several results on homology spheres and homology balls (see \Cref{Homology spheres and homology balls}),
and the Stein factorization of special generic maps (see \Cref{Stein factorization}).

\subsection{Homology spheres and homology balls}\label{Homology spheres and homology balls}
Let $R \neq 0$ be a commutative ring with identity.

\begin{definition}\label{definition homology sphere and ball}
A closed $R$-orientable topological $n$-manifold $P^{n}$ is called an \emph{$R$-homology $n$-sphere} if $\widetilde{H}_{\ast}(P; R) \cong \widetilde{H}_{\ast}(S^{n}; R)$ (where note that $\widetilde{H}_{n}(S^{n}; R) \cong R$ and $\widetilde{H}_{i}(S^{n}; R) = 0$ for $i \neq n$).
A compact $R$-orientable topological $p$-manifold $Q^{p}$ with boundary is called an \emph{$R$-homology $p$-ball} if $\widetilde{H}_{\ast}(Q; R) \cong \widetilde{H}_{\ast}(D^{p}; R)$ (= 0).
\end{definition}

\begin{remark}
Let $P^{n}$ be a closed topological manifold of dimension $n > 0$ such that $H_{n}(P; R) \cong H_{n}(S^{n}; R)$ ($\cong R$).
If $2 R \neq 0$, then $P$ is automatically $R$-orientable.
In fact, if such a manifold $P^{n}$ was not $R$-orientable, then it would follow from Theorem 3.26(b) in \cite{H} that $H_{n}(P; R) \cong \{r \in R; \; 2r = 0\}$, which cannot be isomorphic to $R$ as $R$ has elements of order $>2$ by assumption.
\end{remark}

\begin{remark}\label{remark on orientability with r coefficients}
If a closed topological $n$-manifold $P^{n}$ is orientable, then $P^{n}$ is $R$-orientable for all $R$.
Conversely, if $2 \neq 0$ in $R$, then $R$-orientability of $P^{n}$ implies that $P^{n}$ orientable (see \cite[p. 235]{H}).
\end{remark}

\begin{proposition}\label{proposition homology balls bound homology spheres}
Suppose that $R$ is a principal ideal domain (for example, $R = \mathbb{Z}$ or $R = k$ a field).
If $Q^{p}$ is an $R$-homology $p$-ball of dimension $p \geq 1$, then $\partial Q$ is an $R$-homology $(p-1)$-sphere.
\end{proposition}

\begin{proof}
As $Q^{p}$ is a compact $R$-orientable $p$-manifold, its boundary $\partial Q$ is a closed $R$-orientable $(p-1)$-manifold.
In order to show that $\widetilde{H}_{\ast}(\partial Q; R) \cong \widetilde{H}_{\ast}(S^{p-1}; R)$,
we remove the interior of a $p$-disk $D^{p}$ embedded in the interior of $Q$ to obtain a compact $R$-orientable $p$-manifold $Q'$ with boundary the closed $R$-orientable $(p-1)$-manifold $\partial Q' = \partial Q \sqcup S^{p-1}$.
Using excision and the homotopy axiom for homology, we note that $H_{\ast}(Q', S^{p-1}; R) \cong \widetilde{H}_{\ast}(Q; R) = 0$ because $Q$ is an $R$-homology $p$-ball.
Since $R$ is a principal ideal domain, we can apply the universal coefficient theorem as stated on the bottom of p. 196 in \cite{H} to $G = R$ and the augmented chain complex $C \colon \dots \rightarrow C_{0}(Q', S^{p-1}) \otimes_{\mathbb{Z}} R \rightarrow R \rightarrow 0$ of the pair $(Q', S^{p-1})$ to conclude that $H^{\ast}(Q', S^{p-1}; R) = 0$.
Then, $H_{\ast}(Q', \partial Q; R) = 0$ by Poincar\'{e}-Lefschetz duality (see Theorem 3.43 in \cite[p. 254]{H}).
Finally, from the reduced homology long exact sequences of the pairs $(Q', \partial Q)$ and $(Q', S^{p-1})$ we then see that $\widetilde{H}_{\ast}(\partial Q; R) \cong \widetilde{H}_{\ast}(Q'; R) \cong \widetilde{H}_{\ast}(S^{p-1}; R)$.
\end{proof}

In \Cref{An application in odd dimensions}, we are eventually concerned with the case $R = \mathbb{Q}$, in which we replace the term ``$R$-homology'' by ``rational homology''.

\begin{proposition}\label{proposition r homology implies rational homology}
\begin{enumerate}[(a)]
\item\label{proposition r homology implies rational homology a}
If $P^{n}$ is an $R$-homology $n$-sphere, then $\widetilde{H}_{i}(P; \mathbb{Z})$, $i < n$, are finite abelian groups.
If $P^{n}$ is in addition orientable, then $P^{n}$ is a rational homology $n$-sphere.
\item\label{proposition r homology implies rational homology b}
If $Q^{p}$ is an $R$-homology $p$-ball, then $\widetilde{H}_{i}(Q; \mathbb{Z})$, $i \in \mathbb{Z}$, are finite abelian groups.
If $Q^{p}$ is in addition orientable, then $Q^{p}$ is a rational homology $p$-ball.
\end{enumerate}
\end{proposition}

\begin{proof}
Since $P^{n}$ and $Q^{p}$ are compact manifolds, their integral homology groups are finitely generated in every degree by Corollary A.8 and Corollary A.9 in \cite[p. 527]{H}.
Therefore, by applying the universal coefficient theorem for homology as stated in Theorem 3A.3 in \cite[p. 264]{H} to the augmented chain complex $C \colon \dots \rightarrow C_{0}(P) \rightarrow \mathbb{Z} \rightarrow 0$ of $P$, we conclude from $H_{i}(C; R) = \widetilde{H}_{i}(P; R) = 0$ for $i < n$ that $\operatorname{rank} \widetilde{H}_{i}(P; \mathbb{Z}) = \operatorname{rank} H_{i}(C) = 0$ for $i < n$ because $R \neq 0$.
Similarly, we conclude from $\widetilde{H}_{i}(Q; R) = 0$ for $i \in \mathbb{Z}$ that $\operatorname{rank} \widetilde{H}_{i}(Q; \mathbb{Z}) = 0$ for $i \in \mathbb{Z}$.
Thus, $\widetilde{H}_{i}(P; \mathbb{Z})$, $i < n$, and $\widetilde{H}_{i}(Q; \mathbb{Z})$, $i \in \mathbb{Z}$, are finite abelian groups.
Finally, if $P^{n}$ and $Q^{p}$ are in addition orientable, then, using $H_{\ast}(X; \mathbb{Q}) \cong H_{\ast}(X; \mathbb{Z}) \otimes \mathbb{Q}$ for any space $X$ (see Corollary 3A.6(a) in \cite[p. 266]{H}), it follows that $P^{n}$ is a rational homology $n$-sphere, and $Q^{p}$ is a rational homology $p$-ball.
\end{proof}

\begin{remark}
If a closed $R$-orientable topological $n$-manifold $P^{n}$ of dimension $n > 0$ satisfies $\widetilde{H}_{i}(P; R) = 0$ for $i < n$, then $P$ is an $R$-homology $n$-sphere.
In fact, analogously to the proof of \Cref{proposition r homology implies rational homology}\ref{proposition r homology implies rational homology a} we can show that $\operatorname{rank} \widetilde{H}_{0}(P; \mathbb{Z}) = 0$ because $n > 0$.
Thus, we have $H_{0}(P; \mathbb{Z}) \cong \mathbb{Z}$, and obtain
$$
R \cong \operatorname{Hom}(H_{0}(P; \mathbb{Z}), R) \cong H^{0}(P; R) \cong H_{n}(P; R)
$$
by the universal coefficient theorem for cohomology (Theorem 3.2 in \cite[p.  195]{H}) and Poincar\'{e} duality for $P$ (Theorem 3.30 in \cite[p. 241]{H}).
All in all, $\widetilde{H}_{\ast}(P; R) \cong \widetilde{H}_{\ast}(S^{n}; R)$.
\end{remark}

\subsection{Stein factorization of special generic maps}\label{Stein factorization}
First, let us recall the notion of Stein factorization of an arbitrary continuous map.

\begin{definition}
Let $f \colon X \rightarrow Y$ be a continuous map between topological spaces.
We define an equivalence relation $\sim_{f}$ on $X$ as follows.
Two points $x_{1}, x_{2} \in X$ are called equivalent, $x_{1} \sim_{f} x_{2}$, if there is a point $y \in Y$ such that $x_{1}$ and $x_{2}$ are contained in the same connected component of the fiber $f^{-1}(y)$.
The equivalence relation $\sim_{f}$ on $X$ gives rise to a unique factorization of $f$ of the form

\begin{equation}\label{stein factorization diagram}
\begin{tikzcd}
X \ar{dr}[swap]{q_{f}} \ar{rr}{f} && Y \\
& W_{f}, \ar{ur}[swap]{\overline{f}} &
\end{tikzcd}
\end{equation}
where $W_{f} := X/\sim_{f}$ is the quotient space equipped with the quotient topology, $q_{f} \colon X \rightarrow W_{f}$ is the continuous quotient map, and the map $\overline{f} \colon W_{f} \rightarrow Y$ is continuous.
The diagram (\ref{stein factorization diagram}), or sometimes the space $W_{f}$, is called the \emph{Stein factorization} of $f$.
\end{definition}

Let $f \colon M^{n} \rightarrow \mathbb{R}^{p}$, $1 \leq p < n$, be a special generic map of a connected closed smooth $n$-manifold $M$ into Euclidean $p$-space.
In the following, we recall from \cite{S1} some important properties of the Stein factorization of $f$.

As explained in \cite[p. 267]{S1}, the Stein factorization $W_{f}$ of $f$ can be equipped with the structure of a compact parallelizable smooth $p$-manifold with boundary in such a way that the quotient map $q_{f} \colon M \rightarrow W_{f}$ is a smooth map which satisfies $q_{f}^{-1}(\partial W_{f}) = S(f)$, and restricts to a diffeomorphism $S(f) \cong \partial W_{f}$.
Moreover, it is shown in the proof of Proposition 2.1 in \cite{S1} that $M \setminus S(f)$ is the total space of a smooth (not necessarily linear) $S^{n-p}$-bundle $\pi \colon M \setminus S(f) \rightarrow W_{f} \setminus \partial W_{f}$ over the interior of $W_{f}$.
Furthermore, it is shown there that $M$ is homeomorphic to $\partial \widetilde{E}$, where $\widetilde{E}$ is the total space of the topological $D^{n-p+1}$-bundle $\rho \colon \widetilde{E} \rightarrow W$ associated with the $S^{n-p}$-bundle $\pi| \colon \pi^{-1}(W) \rightarrow W$ that is the restriction of $\pi$ over the closure $W = \overline{W_{f} \setminus C}$ of $W_{f} \setminus C$ in $W_{f}$ for a sufficiently small collar neighborhood $C \cong \partial W_{f} \times [0, 1]$ of $\partial W_{f}$ in $W_{f}$ (compare Proposition 3.1 in \cite{S1}).

Let $R$ be a commutative ring with identity.
Since $\widetilde{E}$ is homotopy equivalent to $W$, and $W \cong W_{f}$ by construction, we have
\begin{equation}\label{homology of total space of disk bundle}
H_{\ast}(\widetilde{E}; R) \cong H_{\ast}(W; R) \cong H_{\ast}(W_{f}; R)
\end{equation}
and
\begin{equation}\label{cohomology of total space of disk bundle}
H^{\ast}(\widetilde{E}; R) \cong H^{\ast}(W; R) \cong H^{\ast}(W_{f}; R).
\end{equation}

From now on, let us assume that $M$ is orientable.
Then, the sphere bundle $\pi$ is orientable in the sense of \cite[p. 442]{H}, and the associated disk bundle $\rho$ is orientable as well.
We conclude that the total space $\widetilde{E}$ of $\rho$ is an orientable compact topological $(n+1)$-manifold.
Since the manifolds $M$, $\widetilde{E}$, and $W_{f}$ are all $R$-orientable by \Cref{remark on orientability with r coefficients}, Poincar\'{e}-Lefschetz duality (see Theorem 3.43 in \cite[p. 254]{H}) implies
\begin{equation}\label{pl for m}
H_{\ast}(M; R) \cong H^{n-\ast}(M; R),
\end{equation}
\begin{equation}\label{pl for e}
H_{\ast}(\widetilde{E}, \partial \widetilde{E}; R) \cong H^{n+1-\ast}(\widetilde{E}; R),
\end{equation}
\begin{equation}\label{pl for wf}
H_{\ast}(W_{f}, \partial W_{f}; R) \cong H^{p-\ast}(W_{f}; R),
\end{equation}
and
\begin{equation}\label{pl for bwf}
H_{\ast}(\partial W_{f}; R) \cong H^{p-1-\ast}(\partial W_{f}; R).
\end{equation}

Analogously to Proposition 3.10 in \cite{S1}, we have a long exact sequence of the form
\begin{equation}\label{long exact sequence}
\begin{tikzcd}
\dots \ar{r} & H_{q+1}(M; R) \ar{r} & H_{q+1}(W_{f}; R) \ar{r} & H^{n-q}(W_{f}; R) & \\
\ar{r} & H_{q}(M; R) \ar{r} & H_{q}(W_{f}; R) \ar{r} & H^{n-q+1}(W_{f}; R) \ar{r} & \dots \\
\dots & & & & \\
\dots \ar{r} & H_{1}(M; R) \ar{r} & H_{1}(W_{f}; R) \ar{r} & H^{n}(W_{f}; R) \ar{r} & 0.
\end{tikzcd}
\end{equation}
(In order to derive (\ref{long exact sequence}), we start with the homology long exact sequence of the pair $(\widetilde{E}, \partial \widetilde{E})$.
Then, we make the replacements $H_{q}(\partial \widetilde{E}; R) \cong H_{q}(M; R)$ by using that $M$ is homeomorphic to $\partial \widetilde{E}$, $H_{q}(\widetilde{E}; R) \cong H_{q}(W_{f}; R)$ by (\ref{homology of total space of disk bundle}), and $H_{q}(\widetilde{E}, \partial \widetilde{E}; R) \cong H^{n+1-q}(\widetilde{E}; R) \cong H^{n+1-q}(W_{f}; R)$ by using (\ref{pl for e}) and (\ref{cohomology of total space of disk bundle}).
The right end of (\ref{long exact sequence}) has the claimed form because $M$ and $W_{f}$ are both connected so that the map $H_{0}(\partial \widetilde{E}; R) \rightarrow H_{0}(\widetilde{E}; R)$ is an isomorphism.)

Next, we note that
\begin{equation}\label{vanishing result}
H^{q}(W_{f}; R) \stackrel{(\ref{pl for wf})}{\cong} H_{p-q}(W_{f}, \partial W_{f}; R) = 0, \qquad q \geq p,
\end{equation}
where $H_{0}(W_{f}, \partial W_{f}; R) = 0$ holds because $W_{f}$ is connected as the image of the connected space $M$ under the surjective Stein factorization $q_{f} \colon M \rightarrow W_{f}$.
Thus, using (\ref{vanishing result}) and (\ref{long exact sequence}), we conclude that
\begin{equation}\label{equality interval}
H_{q}(M; R) \cong H_{q}(W_{f}; R), \qquad q \leq n-p.
\end{equation}

\section{Proof of \Cref{theorem homology spheres}}\label{proof of main result}
Our proof is almost identical to the proof of Proposition 4.1 in \cite{S1}.
However, we have to assure that it still works under the weaker assumptions (we use coefficients in $R$ instead of $\mathbb{Z}$, and make no assumptions about fundamental groups).
\par\bigskip

Let us first suppose that the Stein factorization $W_{f}$ of $f$ is an $R$-homology $p$-ball, with $R$ being a principal ideal domain.
Since $M$ is orientable by assumption and $W_{f}$ is orientable as a parallelizable manifold, the total space $\widetilde{E}$ of the disk bundle $\rho$ is an $R$-orientable compact topological $(n+1)$-manifold.
Moreover, we have
\begin{equation*}
H_{\ast}(\widetilde{E}; R) \stackrel{(\ref{homology of total space of disk bundle})}{\cong} H_{\ast}(W_{f}; R) \cong H_{\ast}(D^{p}; R) \cong H_{\ast}(D^{n+1}; R).
\end{equation*}
Thus, $\widetilde{E}$ is an $R$-homology $(n+1)$-ball.
Hence, using that $R$ is a principal ideal domain, we conclude from \Cref{proposition homology balls bound homology spheres} that $\partial \widetilde{E}$ is an $R$-homology $n$-sphere.
As $M$ is homeomorphic to $\partial \widetilde{E}$, $M$ is an $R$-homology $n$-sphere as well.
\par\bigskip

Conversely, we suppose that $M$ is an $R$-homology $n$-sphere, so that
\begin{equation}\label{assumption homology sphere}
\widetilde{H}_{q}(M; R) \cong \widetilde{H}_{q}(S^{n}; R) = 0, \qquad q < n.
\end{equation}
Combining (\ref{assumption homology sphere}) with the long exact sequence (\ref{long exact sequence}), we obtain
\begin{equation}\label{homology and cohomology of wf}
H_{q}(W_{f}; R) \cong H^{n-q+1}(W_{f}; R), \qquad 0 < q < n.
\end{equation}
Next, we observe that
\begin{equation}\label{induction basis}
\widetilde{H}_{q}(W_{f}; R) = 0, \qquad q \leq n-p+1,
\end{equation}
which follows for $q \leq n-p$ from (\ref{equality interval}) and (\ref{assumption homology sphere}), and for $q = n-p+1$ ($> 0$) (provided that $q<n$)
from (\ref{homology and cohomology of wf}) and (\ref{vanishing result}).

In the following, we show by induction on $q$ that $\widetilde{H}_{q}(W_{f}; R) = 0$ for all $q \leq p$.
(Then, it follows immediately that the compact $R$-orientable $p$-manifold $W_{f}$ is an $R$-homology $p$-ball, where note that $W_{f}$ is orientable as a parallelizable manifold.)
For this purpose, we fix $n-p+2 \leq q \leq p$, and suppose that we have already shown
\begin{equation}\label{induction hypothesis}
\widetilde{H}_{i}(W_{f}; R) = 0, \qquad i \leq q-1.
\end{equation}
In the following, we have to show that $H_{q}(W_{f}; R) = 0$.
(Note that (\ref{induction basis}) is the basis $q=n-p+2$ of the induction.)
Since $0 < q < n$, we have
\begin{equation}\label{start of induction}
H_{q}(W_{f}; R) \stackrel{(\ref{homology and cohomology of wf})}{\cong} H^{n-q+1}(W_{f}; R) \stackrel{(\ref{pl for wf})}{\cong} H_{p-n+q-1}(W_{f}, \partial W_{f}; R) \cong \widetilde{H}_{p-n+q-2}(\partial W_{f}; R),
\end{equation}
where the last isomorphism is a connecting homomorphism in the reduced homology long exact sequence of the pair $(W_{f}, \partial W_{f})$, and is an isomorphism because $\widetilde{H}_{p-n+q-1}(W_{f}; R) = 0 = \widetilde{H}_{p-n+q-2}(W_{f}; R)$ by induction hypothesis (\ref{induction hypothesis}) (note that $p-n+q-1 \leq q-1$ because $p < n$).
If $q = n-p+2$, then we obtain as desired
\begin{equation}
H_{q}(W_{f}; R) \stackrel{(\ref{start of induction})}{\cong} \widetilde{H}_{0}(\partial W_{f}; R) = 0,
\end{equation}
where the last equality holds because the number of connected components of $\partial W_{f} = S(f)$ is at most $1 + \operatorname{rank} H_{p-1}(M; \mathbb{Z}) = 1$ by Proposition 3.15 in \cite{S1}, where note that $M$ is orientable.
(Here, $\operatorname{rank} H_{p-1}(M; \mathbb{Z}) = 0$ holds by \Cref{proposition r homology implies rational homology}\ref{proposition r homology implies rational homology a} because $1 < p < n$.)
If $q > n-p+2$, then
\begin{equation}\label{induction second case}
H_{q}(W_{f}; R) \stackrel{(\ref{start of induction})}{\cong} H_{p-n+q-2}(\partial W_{f}; R) \stackrel{(\ref{pl for bwf})}{\cong} H^{n-q+1}(\partial W_{f}; R).
\end{equation}
Since $\partial W_{f} \cong S(f)$, and $S(f)$ is a closed subset of the closed orientable $n$-manifold $M^{n}$, we have $H^{n-q+1}(\partial W_{f}; R) \cong H_{q-1}(M, M \setminus S(f); R)$ by Poincar\'{e}-Lefschetz duality as stated in Corollary 8.4 in \cite[p. 352]{bre}, where note that the \v{C}ech cohomology can be replaced by singular cohomology because $S(f)$ is a manifold.
Next, we observe that $H_{q-1}(M, M \setminus S(f); R) \cong H_{q-2}(M \setminus S(f); R)$, which is a connecting homomorphism in the reduced homology long exact sequence of the pair $(M, M \setminus S(f))$, and is an isomorphism because $\widetilde{H}_{q-1}(M; R) = 0 = \widetilde{H}_{q-2}(M; R)$ by (\ref{assumption homology sphere}), where $q-1 < n$ because $q \leq p < n$.
Altogether, we have shown that
\begin{equation}\label{induction alexander}
H_{q}(W_{f}; R) \stackrel{(\ref{induction second case})}{\cong} H^{n-q+1}(\partial W_{f}; R) \cong H_{q-1}(M, M \setminus S(f); R) \cong H_{q-2}(M \setminus S(f); R).
\end{equation}

Now we recall that there is a smooth $S^{n-p}$-bundle $\pi \colon M \setminus S(f) \rightarrow \operatorname{int} W_{f}$ over the interior of the Stein factorization $W_{f}$.
We also recall that $\pi$ is an orientable sphere bundle in the sense of \cite[p. 442]{H} because $M^{n}$ is orientable.
Let $D(\pi) \colon E' \rightarrow \operatorname{int} W_{f}$ be the orientable topological $D^{n-p+1}$-bundle associated with $\pi$.
Then, the Thom isomorphism yields $H^{\ast}(W_{f}; \mathbb{Z}) \cong H^{\ast + (n-p+1)}(E', M \setminus S(f); \mathbb{Z})$ (see Corollary 4D.9 in \cite[p. 441]{H}).
Consequently, $H_{\ast}(W_{f}; \mathbb{Z}) \cong H_{\ast + (n-p+1)}(E', M \setminus S(f); \mathbb{Z})$ by Corollary 3.3 in \cite[p. 196]{H}.
Then, Corollary 3A.4 in \cite[p. 264]{H} yields
\begin{equation}\label{thom isomorphism}
H_{\ast}(W_{f}; R) \cong H_{\ast + (n-p+1)}(E', M \setminus S(f); R).
\end{equation}

Let us consider the following part of the homology long exact sequence of the pair $(E', M \setminus S(f))$:
\begin{equation}\label{long exact sequence end of proof}
H_{q-1}(E', M \setminus S(f); R) \rightarrow H_{q-2}(M \setminus S(f); R) \rightarrow H_{q-2}(E'; R).
\end{equation}
Since $E'$ is homotopy equivalent to $W_{f}$, we have $H_{q-2}(E'; R) \cong H_{q-2}(W_{f}; R) = 0$ by induction hypothesis (\ref{induction hypothesis}), where note that $q-2 > 0$ because $q > n-p+2 > 2$.
Furthermore, we have
\begin{equation*}
H_{q-1}(E', M \setminus S(f); R) \stackrel{(\ref{thom isomorphism})}{\cong} H_{p-n+q-2}(W_{f}; R) \stackrel{(\ref{induction hypothesis})}{=} 0,
\end{equation*}
where we can apply the induction hypothesis because $0 < q-n+p-2 < q$.
All in all, we obtain
\begin{equation*}
H_{q}(W_{f}; R) \stackrel{(\ref{induction alexander})}{\cong} H_{q-2}(M \setminus S(f); R) \stackrel{(\ref{long exact sequence end of proof})}{=} 0.
\end{equation*}

This completes the proof of \Cref{theorem homology spheres}.

\section{An application in odd dimensions}\label{An application in odd dimensions}

For the proof of our application \Cref{proposition torsion condition} we need the following
\begin{lemma}\label{lemma square number}
Let $\dots \rightarrow A_{i-1} \rightarrow A_{i} \rightarrow A_{i+1} \rightarrow \dots$ be a long exact sequence of finite abelian groups such that $A_{i} = 0$ for almost all $i \in \mathbb{Z}$.
If $|A_{-i}| = |A_{i}|$ for all $i \in \mathbb{Z}$, then $|A_{0}| = k^{2}$ for some integer $k$.
\end{lemma}

\begin{proof}
Every map $\alpha_{i} \colon A_{i} \rightarrow A_{i+1}$ of the given long exact sequence gives rise to a short exact sequence
\begin{equation*}
0 \rightarrow \operatorname{ker}(\alpha_{i}) \rightarrow A_{i} \rightarrow \operatorname{im}(\alpha_{i}) \rightarrow 0
\end{equation*}
of finite abelian groups.
Thus, we have
\begin{equation}\label{product factorization}
|A_{i}| = |\operatorname{ker}(\alpha_{i})| \cdot |\operatorname{im}(\alpha_{i})|, \qquad i \in \mathbb{Z}.
\end{equation}
Using that $\operatorname{ker}(\alpha_{i}) = \operatorname{im}(\alpha_{i-1})$ by exactness of the given sequence, we have
\begin{equation*}
\prod_{i \text{ odd}}|A_{i}| \stackrel{(\ref{product factorization})}{=} \prod_{i \text{ odd}}|\operatorname{im}(\alpha_{i-1})| \cdot |\operatorname{ker}(\alpha_{i+1})| = \prod_{i \text{ even}}|\operatorname{im}(\alpha_{i})| \cdot |\operatorname{ker}(\alpha_{i})| \stackrel{(\ref{product factorization})}{=} \prod_{i \text{ even}} |A_{i}|,
\end{equation*}
where note that the products are finite because $A_{i} = 0$ for almost all $i \in \mathbb{Z}$.

If $|A_{-i}| = |A_{i}|$ for all $i \in \mathbb{Z}$, then we can write
\begin{equation*}
|A_{0}| = \frac{\prod_{i \text{ odd}} |A_{i}|}{\prod_{0 \neq i \text{ even}} |A_{i}|} = \frac{(\prod_{j \geq 0} |A_{2j+1}|)^{2}}{(\prod_{j \geq 1} |A_{2j}|)^{2}}.
\end{equation*}
Thus, the positive integers $a = |A_{0}|$, $x = \prod_{j \geq 0} |A_{2j+1}|$ and $y = \prod_{j \geq 1} |A_{2j}|$ satisfy $ay^{2} = x^{2}$.
By comparing the exponents of prime numbers in the prime factorizations of $a$, $x$, and $y$, we conclude that $|A_{0}| = k^{2}$ for some integer $k$.
\end{proof}

For the rest of this paper, let $M$ be a connected closed smooth manifold of dimension $n \geq 1$.

The main result of this section is the following

\begin{proposition}\label{proposition torsion condition}
Let $f \colon M^{n} \rightarrow \mathbb{R}^{p}$ ($1 \leq p < n$) be a special generic map.
If $M^{n}$ is a rational homology $n$-sphere of odd dimension $n = 2k+1 \geq 5$, then the cardinality of the finite abelian group $H_{k}(M; \mathbb{Z})$ is the square of an integer.
\end{proposition}

\begin{proof}
Since rational homology spheres are orientable by \Cref{remark on orientability with r coefficients}, we conclude from \Cref{theorem homology spheres} that the Stein factorization $W_{f}$ is a rational homology $p$-ball.
By \Cref{proposition r homology implies rational homology}, $\widetilde{H}_{i}(M; \mathbb{Z})$, $i < n$, and $\widetilde{H}_{i}(W_{f}; \mathbb{Z})$, $i \in \mathbb{Z}$, are finite abelian groups.
Hence, taking $C$ to be the augmented chain complexes of $M$ and $W_{f}$ in Corollary 3.3 in \cite[p. 196]{H}, we obtain
\begin{equation}\label{homology cohomology of m}
\widetilde{H}^{i}(M; \mathbb{Z}) \cong \widetilde{H}_{i-1}(M; \mathbb{Z}), \qquad i < n,
\end{equation}
and
\begin{equation}\label{homology cohomology of wf}
\widetilde{H}^{i}(W_{f}; \mathbb{Z}) \cong \widetilde{H}_{i-1}(W_{f}; \mathbb{Z}), \qquad i \in \mathbb{Z},
\end{equation}
respectively.
Poincar\'{e} duality for $M^{n}$ yields
\begin{equation}\label{homology cohomology poincare of m}
H_{i}(M; \mathbb{Z}) \cong H^{n-i}(M; \mathbb{Z}) \stackrel{(\ref{homology cohomology of m})}{\cong} H_{n-i-1}(M; \mathbb{Z}), \qquad 1 \leq i \leq n-2.
\end{equation}

In view of $H^{n-1}(W_{f}; \mathbb{Z}) = 0$ (see (\ref{vanishing result}) applied for $q = n-1 \geq p$) and (\ref{homology cohomology of wf}), the long exact sequence (\ref{long exact sequence}) takes for $n \geq 5$ and $R = \mathbb{Z}$ the form
\begin{equation}\label{long exact sequence torsion}
\begin{tikzcd}
 & & 0 = H_{n-2}(W_{f}; \mathbb{Z}) \ar{r} & H_{2}(W_{f}; \mathbb{Z}) & \\
\ar{r} & H_{n-3}(M; \mathbb{Z}) \ar{r} & H_{n-3}(W_{f}; \mathbb{Z}) \ar{r} & H_{3}(W_{f}; \mathbb{Z}) \ar{r} & \dots \\
\dots & & & & \\
\dots \ar{r} & H_{q+1}(M; \mathbb{Z}) \ar{r} & H_{q+1}(W_{f}; \mathbb{Z}) \ar{r} & H_{n-q-1}(W_{f}; \mathbb{Z}) & \\
\ar{r} & H_{q}(M; \mathbb{Z}) \ar{r} & H_{q}(W_{f}; \mathbb{Z}) \ar{r} & H_{n-q}(W_{f}; \mathbb{Z}) \ar{r} & \dots \\
\dots & & & & \\
\dots \ar{r} & H_{2}(M; \mathbb{Z}) \ar{r} & H_{2}(W_{f}; \mathbb{Z}) \ar{r} & H_{n-2}(W_{f}; \mathbb{Z}) = 0. &
\end{tikzcd}
\end{equation}
By assumption, $n = 2k+1$ is odd.
Writing $\dots \rightarrow A_{i-1} \rightarrow A_{i} \rightarrow A_{i+1} \rightarrow \dots$ with $A_{0} = H_{k}(M; \mathbb{Z})$ for the above exact sequence (\ref{long exact sequence torsion}), our claim will follow from \Cref{lemma square number} once we show that $|A_{-i}| = |A_{i}|$ for all integers $i > 0$.
If $i \equiv 1 \; (\operatorname{mod} 3)$, say $i=3s+1$, then we see that $A_{i} = H_{k-s}(W_{f}; \mathbb{Z}) = A_{-i}$ for $k-s \geq 2$, and $A_{i} = 0 = A_{-i}$ for $k-s < 2$.
If $i \equiv 2 \; (\operatorname{mod} 3)$, say $i=3s+2$, then we see that $A_{i} = H_{k+1+s}(W_{f}; \mathbb{Z}) = A_{-i}$ for $k+1+s \leq n-2$, and $A_{i} = 0 = A_{-i}$ for $k+1+s > n-2$.
If $i \equiv 0 \; (\operatorname{mod} 3)$, say $i=3s$, then we see by means of (\ref{homology cohomology poincare of m}) that $A_{i} = H_{k-s}(M; \mathbb{Z}) \cong H_{k+s}(M; \mathbb{Z}) = A_{-i}$ for $k-s \geq 2$, and $A_{i} = 0 = A_{-i}$ for $k-s < 2$.
All in all, we have shown that $|A_{-i}| = |A_{i}|$ for all integers $i > 0$, which completes the proof of \Cref{proposition torsion condition}.
\end{proof}

\begin{remark}\label{remark non-sufficiency}
Our homological condition in \Cref{proposition torsion condition} is in general not sufficient for a rational homology sphere $M$ of odd dimension $n \geq 5$ to admit a special generic map into $\mathbb{R}^{p}$ for some $1 \leq p < n$.
In fact, the real projective space $\mathbb{R}P^{5}$ satisfies $H_{2}(\mathbb{R}P^{5}; \mathbb{Z}) = 0$, but there does not exist a special generic map $f \colon \mathbb{R}P^{5} \rightarrow \mathbb{R}^{p}$ for any $1 \leq p < 5$.
(Otherwise, the universal cover $\pi \colon S^{5} \rightarrow \mathbb{R}P^{5}$ would induce a $2$-sheeted covering $W_{f \circ \pi} \rightarrow W_{f}$ of Stein factorizations of the special generic maps $f \circ \pi$ and $f$ by Proposition 2.6 in \cite{hara}.
Thus, the space $W_{f \circ \pi}$ would have even Euler characteristic $\chi(W_{f \circ \pi}) = 2 \chi(W_{f})$ while being contractible by \Cref{theorem homotopy spheres}.)
\end{remark}

\begin{remark}
Suppose that $n = 4l + 1$ for some integer $l \geq 1$.
After Seifert \cite{sei}, the linking form $b \colon TH_{2l}(N; \mathbb{Z}) \times TH_{2l}(N; \mathbb{Z}) \rightarrow \mathbb{Q}/\mathbb{Z}$ on the torsion subgroup of the homology group $H_{2l}(N; \mathbb{Z})$ of a closed oriented topological $n$-manifold $N$ is a nondegenerate skew-symmetric bilinear form.
Wall has shown (see Theorem 3 in \cite{wall}) that
\begin{align}\label{alternatives}
TH_{2l}(N; \mathbb{Z}) \cong \begin{cases}
H \oplus H, \qquad &\text{if } b(x, x) = 0 \text{ for all } x \in TH_{2l}(N; \mathbb{Z}), \\
H \oplus H \oplus \mathbb{Z}/2 \mathbb{Z}, \qquad &\text{else},
\end{cases}
\end{align}
for a suitable finite abelian group $H$.
Thus, if $M$ is a rational homology $n$-sphere that admits a special generic map into $\mathbb{R}^{p}$ for some $1 \leq p < n$, then \Cref{proposition torsion condition} implies that the first alternative holds for $H_{2l}(M; \mathbb{Z})$ in (\ref{alternatives}).
\end{remark}

Next, we show that the homological condition in \Cref{proposition torsion condition} is generally optimal.

\begin{proposition}\label{proposition realization of torsion group}
Let $m > 0$ be an integer.
There exist an integer $k > 1$ and a $(k-1)$-connected rational homology $(2k+1)$-sphere $M$ with $|H_k(M; \mathbb{Z})| = m^{2}$ which admits a special generic map into $\mathbb{R}^{p}$ for some $1 \leq p<2k+1$.
\end{proposition}
\begin{proof}
Let $W^{p}$ be a compact parallelizable smooth $p$-manifold with boundary.
If $n > p > 0$, then by Proposition 2.1 in \cite{S1}, there exists a special generic map $f \colon M^{n} \rightarrow \mathbb{R}^{p}$, where the closed smooth $n$-manifold $M$ is diffeomorphic to the boundary of the product $W \times D^{n-p+1}$ (after smoothing the corners), and the Stein factorization of $f$ is diffeomorphic to $W$.

Now, we suppose in addition that $n=2k+1$ for some integer $k>1$, and that $W$ is a simply connected rational homology $p$-ball whose only non-vanishing integral homology group in positive degree is $H_{k}(W; \mathbb{Z}) \cong \mathbb{Z}/m\mathbb{Z}$.
Then, it follows from \Cref{theorem homology spheres} that $M$ is a rational homology $n$-sphere, and $M$ is simply connected by Proposition 3.9 in \cite{S1}.
Moreover, using the assumptions on the homology of $W$ in the long exact sequence (\ref{long exact sequence torsion}) from the proof of \Cref{proposition torsion condition}, we see that $H_{q}(M; \mathbb{Z}) = 0$ for $2 \leq q \leq k-1$, and that there is a short exact sequence
\begin{equation*}
0 \rightarrow H_{k}(W; \mathbb{Z}) \rightarrow H_{k}(M; \mathbb{Z}) \rightarrow H_{k}(W; \mathbb{Z}) \rightarrow 0.
\end{equation*}
Thus, $M$ is $(k-1)$-connected by the Hurewicz theorem, and we have $|H_{k}(M; \mathbb{Z})| = |H_{k}(W; \mathbb{Z})|^{2} = m^{2}$.
This shows that $M$ will have all the desired properties.

Thus, it remains to construct a manifold $W$ having all of the above properties.
For this purpose, we consider a finite connected simplicial complex $K$ embedded in some $\mathbb{R}^{a}$, and whose only non-vanishing integral homology group in positive degree is $H_{1}(K; \mathbb{Z}) = \mathbb{Z}/m \mathbb{Z}$.
(Such a simplicial complex $K$ can be obtained by an embedded $3$-dimensional lens space $L(m, l) = L_{m}(1, l)$ (compare \Cref{lens spaces} below) with a small open $3$-disk removed.)
Then, by taking $r$-fold suspension, we obtain a finite connected simplicial complex $L$ embedded in $\mathbb{R}^{a+r}$ whose only non-vanishing integral homology group in positive degree is $H_{r+1}(L; \mathbb{Z}) = \mathbb{Z}/m \mathbb{Z}$.
It is well-known that $L$ is the deformation retract of a regular neighborhood $V$ in $\mathbb{R}^{a+r}$ that is a compact smoothly embedded $(a+r)$-manifold (which is in particular parallelizable).
Then, by choosing $r > 0$ so large that $n=2k+1 > p$ with $k=r+1$ and $p=a+r$, the manifold $W=V$ will have all of the desired properties.
(In particular, note that $W$ is simply connected by the Freudenthal suspension theorem.)

This completes the proof of \Cref{proposition realization of torsion group}.
\end{proof}

\begin{remark}
Concerning the choice of an embedded simplicial complex $K \subset \mathbb{R}^{a}$ in the proof of \Cref{proposition realization of torsion group}, we note that $a = 5$ is sufficient because any orientable closed 3-manifold can be embedded in $\mathbb{R}^{5}$ according to a result of Hirsch \cite{hir}.
Moreover, by results of Zeeman \cite{zee} and Epstein \cite{ep}, the $3$-dimensional punctured lens space $L(m, l) \setminus \operatorname{pt}$ can be embedded into $\mathbb{R}^{4}$ if and only if $m$ is odd.
Consequently, in \Cref{proposition realization of torsion group} we can realize all values $k \geq 4$, and also $k = 3$ when $m$ is odd.
We do not know if there exists a special generic map $M^{7} \rightarrow \mathbb{R}^{p}$, $1 \leq p < 7$, where $M$ is a rational homology $7$-sphere such that $|H_{3}(M; \mathbb{Z})|$ is even.
\end{remark}

We conclude with applications of \Cref{proposition torsion condition} to determine the dimension sets of some rational homology spheres.

\begin{example}[lens spaces]\label{lens spaces}
For an integer $m > 1$ and integers $l_{1}, \dots, l_{k+1}$ ($k \geq 0$) relatively prime to $m$, the lens space $L_{m}(l_{1}, \dots, l_{k+1})$ (see e.g. Example 2.43 in \cite[p. 144]{H}) is a closed smooth $(2k+1)$-manifold whose integral homology groups are given by
\begin{align*}
H_{i}(L_{m}(l_{1}, \dots, l_{k+1}); \mathbb{Z}) = \begin{cases}
\mathbb{Z} \qquad &\text{for } i=0, 2k+1, \\
\mathbb{Z}/m\mathbb{Z} \qquad &\text{for } i \text{ odd, } 0 < i < 2k+1, \\
0 \qquad &\text{otherwise}.
\end{cases}
\end{align*}
If $k \geq 3$ is odd, and $m = |H_{k}(L_{m}(l_{1}, \dots, l_{k+1}); \mathbb{Z})|$ is not the square of an integer, then \Cref{proposition torsion condition} implies that $L_{m}(l_{1}, \dots, l_{k+1})$ does not admit a special generic map into $\mathbb{R}^{p}$ for any $1 \leq p < 2k+1$.
Furthermore, \`{E}lia\v{s}berg \cite{E} has shown that there is a special generic map $L_{m}(l_{1}, \dots, l_{k+1}) \rightarrow \mathbb{R}^{2k+1}$ if and only if $L_{m}(l_{1}, \dots, l_{k+1})$ is stably parallelizable.
Hence, we have $S(L_{m}(l_{1}, \dots, l_{k+1})) = \{2k+1\}$ if $L_{m}(l_{1}, \dots, l_{k+1})$ is stably parallelizable, and $S(L_{m}(l_{1}, \dots, l_{k+1})) = \emptyset$ else.
If $m$ is an odd prime and $1 \leq l_{i} \leq m-1$ for all $i$, then it follows from \cite{emss} that $S(L_{m}(l_{1}, \dots, l_{k+1})) = \{2k+1\}$ if and only if $k < m$ and $l_{1}^{2j} + \dots + l_{k+1}^{2j}$ is divisible by $m$ for $j = 1, \dots, \lfloor k/2 \rfloor$, where $\lfloor x \rfloor$ denotes the biggest integer $\leq x$ for a real number $x$.
\end{example}

\begin{example}[linear $S^{3}$-bundles over $S^{4}$]\label{s3 bundles over s4}
As explained in \cite{ce}, fiber bundles over $S^{4}$ with fiber $S^{3}$ and structure group $SO(4)$ are classified by elements of $\pi_{3}(SO(4)) \cong \mathbb{Z} \oplus \mathbb{Z}$.
Moreover, the nontrivial integral homology groups of the total space $M_{m, n}$ corresponding to $(m, n) \in \pi_{3}(SO(4))$ are $H_{0}(M_{m, n}; \mathbb{Z}) \cong H_{7}(M_{m, n}; \mathbb{Z}) \cong \mathbb{Z}$ and $H_{3}(M_{m, n}; \mathbb{Z}) \cong \mathbb{Z}/n \mathbb{Z}$.
We note that $M_{m, n}$ is a rational homology $7$-sphere for $n \neq 0$.
Hence, by \Cref{proposition torsion condition}, $M_{m, n}$ does not admit a special generic map into $\mathbb{R}^{p}$ for any $1 \leq p < 7$ whenever $|n|$ is not the square of an integer.
Moreover, \`{E}lia\v{s}berg \cite{E} has shown that there is a special generic map $M_{m, n} \rightarrow \mathbb{R}^{7}$ if and only if $M_{m, n}$ is stably parallelizable.
According to Wilkens \cite{wil}, this is equivalent to the vanishing of an obstruction $\widehat{\beta} \in H^{4}(M_{m, n}; \pi_{3}(SO)) \cong \mathbb{Z}/ n \mathbb{Z}$.
This obstruction has been determined to be $\widehat{\beta} = \frac{p_{1}}{2}(M_{m, n}) \equiv 2m \; (\operatorname{mod} n)$ in \cite[p. 365]{ce}.
All in all, if $|n|$ is not the square of an integer, then
\begin{align*}
S(M_{m, n}) = \begin{cases}
\{n\}, \qquad &n | 2m, \\
\emptyset, \qquad &\text{else}.
\end{cases}
\end{align*}
\end{example}

\textbf{Acknowledgements.}
The author would like to thank Osamu Saeki for invaluable comments on an early draft of the paper.

\bibliographystyle{amsplain}

\end{document}